\documentclass[11pt]{article}

\usepackage[
backend=biber,
sorting=nyt,
maxbibnames=99
]{biblatex}

\usepackage[verbose=true,letterpaper]{geometry}

\usepackage[utf8]{inputenc} % allow utf-8 input
\usepackage[T1]{fontenc}    % use 8-bit T1 fonts
\usepackage{hyperref}       % hyperlinks
\usepackage{url}            % simple URL typesetting
\usepackage{booktabs}       % professional-quality tables
\usepackage{amsfonts}       % blackboard math symbols
\usepackage{nicefrac}       % compact symbols for 1/2, etc.
\usepackage{microtype}      % microtypography
\usepackage{lipsum}
\usepackage{fancyhdr}       % header
\usepackage{graphicx}       % graphics
\usepackage{amsmath}
\usepackage{amsthm}
\usepackage{xcolor}
\usepackage{comment}
\usepackage{setspace}
\usepackage{cleveref}
\newcommand{\M}{\mathcal{M}}
\newcommand{\C}{\mathcal{C}}
\newcommand{\R}{\mathbb{R}}
\newcommand{\vol}{\mathrm{Vol}}
\newcommand{\Ellipsoid}{\mathcal{E}}
\newcommand{\Rindices}{\R^{\binom{[n]}{2}}}
\newcommand{\nchoosetwo}{\binom{n}{2}}
\newcommand{\unitball}{\textbf{B}_{{[n]\choose 2}}}
\newcommand{\lownerjohn}{L\"{o}wner-John }
\DeclareMathOperator*{\argmax}{argmax}

\newcommand{\columnvecsingleval}[1]{\left(\begin{smallmatrix}
    #1 \\
    \vdots \\
    #1
\end{smallmatrix}\right)}

\newtheorem{theorem}{Theorem}[section]
\newtheorem{lemma}[theorem]{Lemma}
\newtheorem{corollary}[theorem]{Corollary}
\newtheorem{fact}[theorem]{Fact}
\newtheorem*{remark}{Remark}
\theoremstyle{definition}
\newtheorem{definition}{Definition}[section]

\author{Raziel Gartsman\thanks{School of Computer Science and Engineering, Hebrew University, Jerusalem 91904, Israel. e-mail: razielg@cs.huji.ac.il.}
                        \and{Nati Linial\thanks{School of Computer Science and Engineering, Hebrew University, Jerusalem 91904, Israel. e-mail: nati@cs.huji.ac.il. Supported in part by an NSF-BSF grant "Global Geometry of Graphs"}}}
\date{}
\title{On the \lownerjohn Ellipsoids of the Metric Polytope}

\begin{document}
\maketitle

\onehalfspacing

\begin{abstract}
The collection of all $n$-point metric spaces of diameter $\le 1$ constitutes
a polytope $\M_n \subset \R^{\nchoosetwo}$, called the {\em Metric Polytope}. In this paper, we
consider the best approximations of $\M_n$ by ellipsoids. We give an exact explicit 
description of the largest volume ellipsoid contained in $\M_n$. When
inflated by a factor of $\Theta(n)$,
this ellipsoid contains $\M_n$. It also turns out that
the least volume ellipsoid containing $\M_n$ is a ball. 
When shrunk by a factor of $\Theta(n)$,
the resulting ball is contained in $\M_n$. We note that the general theorems on such
ellipsoid posit only that the pertinent
inflation/shrinkage factors can be made as small as $O(n^2)$.
\end{abstract}

\section{Introduction}

For an integer $n \ge 3$ we denote by $\Rindices$ the real $\binom{n}{2}$-dimensional space whose coordinates are
indexed by unordered pairs $ij$, $1 \le i \neq j \le n$. Alternatively, we think of a point in $\Rindices$ as a real symmetric $n\times n$
matrix with zeros along the main diagonal. 
This allows us to view an $n$-vertex metric space as a point $d \in \Rindices$, where $d_{ij}$ is the distance between the two vertices $i$, $j$. The set of all $n$-point metric spaces ${\M\C}_n$ is an $\binom{n}{2}$-dimensional cone, called the {\em Metric Cone}. Namely, 
\begin{equation*}
    {\M\C}_n = \left\{d \in \Rindices \mid \text{ for all distinct $1 \le i, j, k \le n$: } d_{ij} \ge 0 \text{~and~} d_{ij} + d_{jk} - d_{ik} \ge 0 \right\}.
\end{equation*}
We denote by $\M_n = {\M\C}_n \cap {\left[ 0, 1 \right]} ^ {\binom{[n]}{2}}$
the {\em Metric Polytope}, comprised of all $n$-point metric spaces of diameter $\le 1$.
Other versions of this polytope appear as well in the literature (see, e.g. \cite{kozma2021does} or \cite{deza1997geometry, MR2043643}). The study of $\M_n$ has received considerable attention in the literature. A close  relative of the Metric Polytope is the Cut Polytope (see below), also a widely studied geometric object.

It is a basic fact in convex geometry that every convex body $K$ can be well approximated by ellipsoids. Two fundamental results in this area are due to John and to L\"{o}wner.
These theorems say that $K$ can be 'sandwiched' between its maximal volume inscribed ellipsoid and its dilation by a factor of $t$. Likewise for the 
least-volume ellipsoid containing $K$. By convention, $t\Ellipsoid$ denotes the ellipsoid that is obtained from $\Ellipsoid$
as we inflate it by a factor $t>0$ relative to $\Ellipsoid$'s center.

\begin{fact}[\lownerjohn \cite{howard1997john, john1948extremum}] \label{th:lowner_john_dilation}
For every convex body $K \subset \R^m$ there is a unique ellipsoid $\Ellipsoid_{inner}$ of largest volume that is contained in $K$. Moreover,
\[\Ellipsoid_{inner} \subseteq K \subset m\Ellipsoid_{inner}.\]

Also, there is a unique minimal volume ellipsoid $\Ellipsoid_{outer}$, that contains $K$ and
\[\frac{1}{m} \Ellipsoid_{outer} \subset K \subseteq \Ellipsoid_{outer}.\]
\end{fact}

These inner and outer ellipsoids are sometimes referred to as \lownerjohn ellipsoids \cite{henk2012lowner}. Our main results are explicit expressions for the \lownerjohn ellipsoids of $\M_n$, given in  theorems \ref{th:outer_main} and \ref{th:inner_main}:

\begin{theorem}[Outer \lownerjohn Ellipsoid of $\M_n$] \label{th:outer_main}
The least volume ellipsoid $\Ellipsoid_n$ that contains $\M_n \subseteq \Rindices$ is the $\nchoosetwo$-dimensional ball with radius $\frac{1}{2} \sqrt{\nchoosetwo}$ centered at $\columnvecsingleval{\nicefrac{1}{2}}$.

By shrinking $\Ellipsoid_n$ by a factor of $\sqrt{3\nchoosetwo}$ we obtain a ball that is contained in $\M_n$. This factor is tight.
\end{theorem}

\begin{theorem}[Inner \lownerjohn Ellipsoid of $\M_n$] \label{th:inner_main}
The largest volume ellipsoid $\Ellipsoid_n$ that is contained in $\M_n$ is defined as follows. 
Consider the following $\nchoosetwo \times \nchoosetwo$ positive semidefinite matrix $A_n$ and vector $c_n \in \Rindices$
\[\Ellipsoid_n = A_n \unitball + c_n, \]
where $\unitball$ is the $\nchoosetwo$-dimensional unit ball.

Namely,

\begin{align*}
    & c_n = \columnvecsingleval{\delta} \\
    & \text{for all } ij, kl \in \binom{[n]}{2}: {\left[A_n \right]}_{ij, kl} =
        \left\{
        \begin{matrix}
            \alpha & ij = kl \\
            \beta & \left|\{i, j\} \cap \{k, l\}\right| = 1 \\
            \gamma & \{i, j\} \cap \{k, l\} = \emptyset
        \end{matrix}
        \right.
\end{align*}
Here
\begin{alignat*}{2}
    \alpha &= \frac{\sqrt{3} - 1}{2} + o(1)                                                                              && \approx 0.3660 + o(1) \\
    \beta  &= \frac{\sqrt{3} - 1}{2} \cdot (\sqrt{3 - \sqrt{3}} - 1) \cdot \frac{1}{n} + o\left( \frac{1}{n}\right)      && \approx 0.0461 \cdot \frac{1}{n} + o\left( \frac{1}{n}\right) \\
    \gamma &= \frac{\sqrt{3} - 1}{2} \cdot (4 - 2\sqrt{3 - \sqrt{3}}) \cdot \frac{1}{n^2} + o\left( \frac{1}{n^2}\right) && \approx 0.6398 \cdot \frac{1}{n^2} + o\left( \frac{1}{n^2}\right)\\
    \delta &= \frac{\sqrt{3} - 1}{2} \cdot \sqrt{3} + o(1)                                                               && \approx 0.6340 + o(1).
\end{alignat*}

The smallest $r$ for which $\Ellipsoid_{n} \subseteq \M_n \subseteq r\Ellipsoid_{n}$, is between
$\frac{3^{\frac{3}{4}} \sqrt{5 + \sqrt{3}}}{6}n + o(n) \approx 0.99n + o(n)$
and $\sqrt{3\binom{n}{2}} +o(n) \approx 1.22n + o(n)$.
\end{theorem}

\begin{remark}
Note that the shrinkage and inflation parameters in theorems \ref{th:outer_main}, \ref{th:inner_main} are only linear in $n$, whereas
the general Theorem \ref{th:lowner_john_dilation} only yields quadratic factors.
\end{remark}

This paper is organised as follows:
We prove Theorem \ref{th:outer_main} in section \ref{sec:outer}. We then proceed to proving Theorem \ref{th:inner_main} in subsection \ref{subsec:inner_proof}, and discuss some consequences of this result in subsection \ref{subsec:inner_conclusions}.

\subsection{Some preliminaries} \label{sec:preliminaries}

Here we collect several standard facts for reference:
an $m$-dimensional ellipsoid $\Ellipsoid \subset \R^{m}$ is the affine image of the unit ball $A \textbf{B}_{m} + c$, where $A$ is a positive semidefinite matrix and $c$ is a vector. Its geometric properties are given by $A$ and $c$, in the following sense:
\begin{itemize}
    \item $c$ is the ellipsoid's {\em center}.
    \item the ellipsoid's axes are aligned on $A$'s eigenvectors, and have lengths proportional to the corresponding eigenvalues.
    \item $\vol(\Ellipsoid) = (\det A) \cdot \vol(\textbf{B}_{m})$.
\end{itemize}

\section{The Outer \lownerjohn Ellipsoid of $\M_n$} \label{sec:outer}

The following theorem of John characterizes the least-volume ellipsoid $\Ellipsoid$ containing a given convex body.
In view of the comment in Section \ref{sec:preliminaries}, it suffices to consider the special case $\Ellipsoid =B_n$.

\begin{theorem}[John \cite{john1948extremum}] \label{th:john_outer_contact}
Let $K \subset \R^N$ be a convex body where $K \subseteq B_N$. Then the following statements are equivalent:
\begin{itemize}
    \item $B_N$ is the unique least volume ellipsoid containing $K$.
    \item There exist contact points $u_1, ..., u_m \in \partial K \cap \partial B_N$, and positive numbers $\lambda_1, ..., \lambda_m$, with $m \ge N$, such that
    \( \sum_{i=1} ^{m} {\lambda_i u_i} = 0 \) and \(I_n = \sum_{i=1} ^{m} {\lambda_i \left(u_i u_{i}^{t}\right)}\),
    where $I_n$ is the identity matrix.
\end{itemize}
\end{theorem}

We prove Theorem \ref{th:outer_main} by finding the appropriate contact points and the coefficients $\lambda_i$. As it turns out, and as we show now these contact points coincide
with {\em cut metrics} over $[n]$.

\begin{definition}[Cut Metric]
An $n$-point metric $d\in \Rindices$ is called a Cut Metric if there is a set $S\subseteq [n]$ for which $\forall 1 \leq i \neq j\le n: d_{ij} = \left\{\begin{smallmatrix}
1 & |\{i,j\} \cap S|=1 \\
0 & else
\end{smallmatrix}
\right. .$
The cut metric corresponding to a set $S \subseteq [n]$ is denoted by $\delta(S) \in \Rindices$.
\end{definition}

The convex hull of all cut metrics over $[n]$ is called the Cut Polytope $\mathcal{C}_n \subset \Rindices$. We next recall its close relationship with $\M_n$.

\begin{theorem} \label{th:cut_is_vertex}
Every cut metric in $\Rindices$ is a vertex of $\M_n$.
\end{theorem}

\begin{proof}
A point in $\M_n$ is a vertex if and only if it is not the convex combinations of any two distinct points in $\M_n$. Since all cut metrics have only $0$- or $1$-coordinates, they clearly cannot be such combinations.
\end{proof}

We are now ready to prove Theorem \ref{th:outer_main}.

\begin{proof}[Proof of Theorem \ref{th:outer_main}]
There are $m = 2^{n - 1}-1 \ge \nchoosetwo$ cut metrics over $[n]$.
As Theorem \ref{th:cut_is_vertex} shows, every cut metric $\delta(S)$ is a vertex of $\M_n$. 
For a set $S \subseteq [n]$, denote \[\delta^{'}(S) = 2\delta(S) - 1 = \left\{\begin{smallmatrix}
1 & |\{i,j\} \cap S|=1 \\
-1 & else
\end{smallmatrix}
\right.\]
and let $\M_n^{'} = \frac{1}{\frac{1}{2} \sqrt{\nchoosetwo}} \M_n - \frac{1}{2}\columnvecsingleval{1}$.
Note that $\frac{1}{\sqrt{\nchoosetwo}} \delta^{'}(S)$ is a unit vector and a vertex of $\M_n^{'}$.
Denote those vertices as $u_1, ..., u_m$, and set $\forall i:\: \lambda_i = \lambda = \frac{\nchoosetwo}{m}$. For any $1 \le k \neq l \le n$, the following holds:
\[
\left[\sum^m_{i=1}{\lambda_{i}(u_i u_i^T)}\right]_{kl, kl} = \lambda \sum_{\delta(S), S\subseteq [n]} {\left(\frac{1}{\sqrt{\nchoosetwo}}\left[\delta^{'}(S)\right]_{kl} \right)^2} = \lambda \frac{m}{\nchoosetwo} = 1 .
\]
Let us sample a subset $S\subseteq [n]$ and consider the corresponding cut metric $\delta(S)$.
Note that for $\{k,l\} \neq \{p,q\}$, the random variables $\left[\delta^{'}(S)\right]_{kl}$, $\left[\delta^{'}(S)\right]_{pq}$ are independent. Therefore:
\[
    \left[\sum^{m}_{i=1}{\lambda_{i} \left(u_i u_i^T \right)}\right]_{kl, pq} =
    \lambda\sum_{S\subseteq[n]} {\frac{1}{\nchoosetwo} \left[\delta^{'}(S)\right]_{kl} \left[\delta^{'}(S)\right]_{pq}} =
    \frac{\lambda m}{\nchoosetwo} \mathop{\mathbb{E}^2}_{S\subseteq[n]} {\left[\delta^{'}(S)\right]_{kl}}
    = 0 .
\]
Now, $\sum^m_{i=1}{\lambda_{i} u_{i} u_{i}^T} = I_{\nchoosetwo}$ and $\sum^m_{i=1}{\lambda_{i} u_{i}} = 0$. By John's Theorem \ref{th:john_outer_contact}, the minimal volume ellipsoid containing $\M_n^{'}$ is $\unitball$. By applying the inverse of the above affine map, the first part of Theorem \ref{th:outer_main} is obtained.

Finally, denote $p = \columnvecsingleval{\nicefrac{1}{2}}$. Note that $\M_n$ has 2 families of facets.
One expresses the upper bound on the diameter of metric spaces, and each 
such facet is at distance $\frac{1}{2}$ from $p$. The other ones are the triangle inequalities, and each one is at distance $\frac{1}{2\sqrt3}$ from $p$.
Thus, the largest ball centered at $p$ and contained in $\M_n$ has radius $\frac{1}{2\sqrt3}$. The result now directly follows.
\end{proof}

\section{The Inner \lownerjohn Ellipsoid of $\M_n$}

In this section, we discuss the (unique) largest volume ellipsoid $\Ellipsoid_n$ contained in $\M_n$. Our main result is a description of this ellipsoid, given in Theorem \ref{th:inner_main}. The proof is in four parts:
\begin{enumerate}
    \item The fact that $\M_n$ is invariant under the action of the symmetric group
     $S_n$ allows us
     to infer a general form of the structure of $\Ellipsoid_n$. The results of this part (\ref{th:alpha_beta_form}, \ref{th:eigenvals}) apply to any convex body with such symmetry.
    \item An inner \lownerjohn Ellipsoid is the solution to a maximization problem. The symmetry of $\M_n$ and the resulting structure of $\Ellipsoid_n$ mentioned above
    help us solve this optimization problem, see Lemma \ref{th:maximization_problem}.
    \item An asymptotic solution of the problem in Lemma \ref{th:maximization_problem}, yields an explicit description of $\Ellipsoid_n$.
    \item Finally, this description lets us bound the minimal $r$ for which $\M_n \subseteq r\Ellipsoid_n$, proving Theorem \ref{th:inner_main}.
\end{enumerate}

We then continue to discuss some implications of this result. Namely, corollary \ref{th:min_distance} describes the minimal distance possible in a metric space $d \in \Ellipsoid_n$, and Theorem \ref{th:inner_contact_lemma} describes all contact points $\M_n \cap \Ellipsoid_n$.

\subsection{Proof of Theorem \ref{th:inner_main}} \label{subsec:inner_proof}
\begin{lemma} \label{th:alpha_beta_form}
Let $n \ge 3$ be an integer and let $\Ellipsoid_n$ be the largest volume ellipsoid contained in $\M_n$. Then 
\[\Ellipsoid_n = A_n \unitball + c_n\]
where $\unitball$ is the $\nchoosetwo$-dimensional unit ball, $A_n$ is an $\nchoosetwo \times \nchoosetwo$ positive semi-definite matrix
and $c_n \in \Rindices$ is as follows
\begin{align}
    & c_n = \columnvecsingleval{\delta_n} \label{eq:d_rep} \\
    & \text{for all $ij, kl \in \binom{[n]}{2}$: } {\left[A_{n}\right]}_{ij, kl} =
        \left\{
        \begin{matrix}
            \alpha_n & ij = kl \\
            \beta_n  & \left|\{i, j\} \cap \{k, l\}\right| = 1 \\
            \gamma_n & \{i, j\} \cap \{k, l\} = \emptyset
        \end{matrix}
        \right. \label{eq:c_rep}
\end{align}

with $\alpha_n , \beta_n , \gamma_n , \delta_n \in \R$. The specific values of these parameters are determined below.
\end{lemma}

\begin{remark}
For convenience of notation, when $n$ is obvious from the context we may omit the subscript, e.g., write $\alpha$ instead of $\alpha_n$.
\end{remark}

As stated above, the main idea of the proof builds on $\M_n$'s invariance under the action of $S_n$. The
permutation $\sigma \in S_n$ acts on $x \in \Rindices$ via
\[{\left[\sigma (x)\right]}_{ij} = x_{\sigma (i) \sigma (j)} .\] 
This action extends to subsets $B \subseteq \Rindices$ by $\sigma(B) = \{\sigma(x) \mid x \in B\}$, and to matrices $M$ by $[\sigma(M)]_{ij, kl} = M_{\sigma(i)\sigma(j), \sigma(k)\sigma(l)}$.
\begin{proof}[Proof of lemma \ref{th:alpha_beta_form}]
We wish to show that $\Ellipsoid_n$ is $S_n$-invariant.
Indeed, $\Ellipsoid_n \subseteq \M_n$ is the \textbf{unique} largest volume ellipsoid contained in $\M_n$. The conclusion
follows by acting with $\sigma\in S_n$ on the two sides of this inclusion, using the facts that
$\M_n$ is $\sigma$-invariant, and the action of $\sigma$ is volume-preserving.

To prove \eqref{eq:d_rep}, note
that $c_n$, the center of $\Ellipsoid_n$, is invariant under any $\sigma\in S_n$.

To prove \eqref{eq:c_rep}, let $D \in PSD_{\binom{[n]}{2}}$ be such that
\[
 \Ellipsoid_n = D \unitball + c_n
\]
and let
\[
A_n = \frac{1}{|S_n |} \sum_{\tau \in S_n} \tau(D)  \in PSD_{\binom{[n]}{2}}.
\]

The conclusion follows, since $\Ellipsoid_n = A_n \unitball + c_n$, and since $A_n$ is $S_n$-invariant. 
\end{proof}

\begin{theorem} \label{th:eigenvals}
Let $n \ge 3$ be an integer and let $A_n$ be defined as above. Then $A_n$ has (at most) three distinct eigenvalues:
\begin{alignat}{4}
    &\lambda_1 & =&\alpha & + 2(n-2)&\beta & +\binom{n-2}{2}&\gamma \quad \text{with multiplicity} \quad 1 \label{eq:eigens_expr_first} \\
    &\lambda_2 & =&\alpha &  + (n-4)&\beta &          -(n-3)&\gamma \quad \text{with multiplicity} \quad n-1 \\
    &\lambda_3 & =&\alpha &      - 2&\beta & +              &\gamma \quad \text{with multiplicity} \quad \binom{n}{2} - n \label{eq:eigens_expr_last}
\end{alignat}
\end{theorem}

\begin{proof}

In this proof we explicitly determine all of $A_n$'s eigenvectors and their corresponding eigenvalues.

We show first that $\columnvecsingleval{1}$ is an eigenvector with eigenvalue $\lambda_1$. Indeed, for any $ij \in \binom{[n]}{2}$:
\[
\left[A_n \cdot \columnvecsingleval{1} \right]_{ij} = \sum_{kl}{[A_n]_{ij,kl}} = \alpha + 2(n-2)\beta + \binom{n-2}{2}\gamma ,
\]
as stated.

For any $1 \le i \le n$ define $s(i)$ by
\([s(i)]_{kl} = \left\{
\begin{smallmatrix}
    1 & i \in \{k, l\} \\
    0 & otherwise
\end{smallmatrix}
\right.\). Note that $s(1) - s(2)$ is an eigenvector of $A_n$ with eigenvalue $\lambda_2$:
\begin{align*}
    [A_{n} \cdot (s(1) - s(2))]_{ij} & = \sum_{kl}{[A_{n}]_{ij, kl} [s(1)-s(2)]_{kl}} = \sum_{k \ge 3}{[A_n]_{ij, 1k}} - \sum_{k \ge 3}{[A_n]_{ij, 2k}} = \\
    & = \left\{
    \begin{matrix}
          \alpha + (n-4)\beta - (n-3)\gamma\ & \{i, j\} \cap \{1, 2\} = \{1\} \\
        - \alpha - (n-4)\beta + (n-3)\gamma\ & \{i, j\} \cap \{1, 2\} = \{2\} \\
        0 & otherwise
    \end{matrix}
    \right. \\
    & = [\lambda_2 \cdot (s(1) - s(2))]_{ij} .
\end{align*}
Due to the $S_n$-symmetry of $A_n$, every vector of the form $s(i) - s(j)$ with $1 \le i\neq j \le n$
is an eigenvector of $A_n$, with eigenvalue $\lambda_2$. The multiplicity of $\lambda_2$ is at least $n-1$, because the $n-1$ vectors
$s(1) - s(2),\ldots, s(1)-s(n)$ are linearly independent.

For $1\le r \neq s\le n$ we denote by $e_{rs}$ the $\binom{n}{2}$-dimensional vector with a single $1$-entry at position $rs$ and $0$-entries elsewhere.
For any four \textbf{distinct} indices $1\le i, j, k, l \le n$, let $\xi(i, j, k, l) = e_{ij} - e_{jk} + e_{kl} - e_{il}$. 
Notice that the function $\xi$ is not symmetric, so while $i,j,k,l$ are only required to be a $4$-tuple with no reference to their ordering,
this symmetry is broken due to the definition of the function $\xi$. The following holds:

\begin{multline*}
    [A_n \cdot \xi(1, 2, 3, 4)]_{ij} = [A_n]_{ij, 12} - [A_n]_{ij, 23} + [A_n]_{ij, 34} - [A_n]_{ij, 14} = \\
    =
    \left\{
    \begin{matrix}
          \gamma - \gamma + \gamma - \gamma & = 0           & i, j > 4 \\
          \beta  - \beta  + \gamma - \gamma & = 0           & \min (i, j) \le 4 < \max (i, j) \\
          \beta  - \beta  + \beta  - \beta  & = 0           & ij \in \{13, 24\} \\
          \alpha - \beta  + \gamma - \beta  & = \lambda_3   & ij \in \{12, 34\} \\
          \beta  - \alpha + \beta  - \gamma & = - \lambda_3 & ij \in \{23, 14\} \\
    \end{matrix}
    \right.
\end{multline*}

In other words, $A_n \cdot \xi(1, 2, 3, 4) = \lambda_3 \cdot \xi(1, 2, 3, 4)$. Again, by symmetry $\xi(i, j, k, l)$
is an eigenvector of $A_n$ with eigenvalue $\lambda_3$ for every distinct $1\le i, j, k, l \le n$.
In order to show that the multiplicity of $\lambda_3$ is at least $\nchoosetwo - n$ we exhibit this
many linearly independent vectors in this eigenspace. Let
\[
    B_1=\{\xi(1, 2, i, j) \mid 3 \le i < j\}, ~~ B_2=\{\xi(1, 2, 3, j) - \xi(1, 2, j, 3) \mid j \ge 4\}.
\]
Since $|B_1| = \binom{n-2}{2}$ and $|B_2|=n-3$ and $\binom{n-2}{2}+n-3= \binom{n}{2} - n$, it suffices to show that the vectors in $B_1 \sqcup B_2$
are linearly independent. For every vector $w\in B_1 \sqcup B_2$ let $\ell(w)$ be the lexicographically last coordinate in $w$'s support.
Our claim follows from the easily verifiable fact that $\ell(w_1)\neq \ell(w_2)$ for every two distinct vectors $w_1, w_2\in B_1 \sqcup B_2$.

It follows that we have accounted for all of $A_n$'s eigenvalues, since we have exhibited $\nchoosetwo = 1 + (n-1) + \left(\nchoosetwo - n\right)$ linearly independent eigenvectors.
\end{proof}

With the above notation, we determine $\alpha_{n}, \beta_{n}, \gamma_{n}$ and $\delta_{n}$ by minimizing the relevant determinant (see \ref{sec:preliminaries} for 
the relevant background).
This leads to an optimization problem that we proceed to solve. 

\begin{lemma} \label{th:maximization_problem}
The values of $\alpha_{n}, \beta_{n}, \gamma_{n}$ and $\delta_{n}$ are the solution of the following optimization problem (with $\lambda_{1}, \lambda_{2}, \lambda_{3}$ as in Theorem \ref{th:eigenvals}):

\begin{align}
    \text{maximize} \quad & \log (\lambda_1 \cdot \lambda_{2}^{n-1} \cdot \lambda_{3}^{\binom{n}{2} - n}) \\
    \text{s.t.} \quad     & \lambda_1, \lambda_2, \lambda_3 > 0 \label{eq:final_con_psd} \\
                          & 3\alpha^2 - 4\alpha\beta + 4(n-2)\beta^2 -4(n-3)\beta\gamma + \left[\binom{n}{2}-3\right] \gamma^2 \le \delta^2 \label{eq:final_con_1} \\ 
                          & \alpha^2 + 2(n-2)\beta^2 + \binom{n-2}{2}\gamma^2 \le {(1-\delta)}^2 \label{eq:final_con_2}
\end{align}
\end{lemma}

\begin{proof}
Let $\M_n=\{x \mid \forall 1 \le i \le m: \: a_{i}^{T} x \le b_i \}$. Furthermore, let $\Ellipsoid = A \unitball + c$ be an ellipsoid,
where $A$ is a positive semidefinite matrix. The inclusion $\Ellipsoid\subseteq \M_n$ holds if and only if $a_i^T x \le b_i$ for every $x \in \Ellipsoid$ and every index $i$.
Equivalently, $a_{i}^{T} Au + a_{i}^{T} d \le b_i$ for every $n$-dimensional unit vector $u$. Since
$\displaystyle \max_{u \in \unitball}{a_{i}^T Au} = \|a_{i}^{T} A\|_2$, we have that $A \unitball + c \subseteq \M_n$ if and only if for all $1 \le i \le m$ it holds that $\|a_{i}^{T} A\|_2 \le b_i - a_{i}^{T} c$.

In addition, $\vol(\Ellipsoid)$ strictly increases with $\log \det A$, hence the unique maximal volume ellipsoid contained in $\M_n$ must be the single maximizer of the following problem:
\begin{align}
    \text{maximize} \quad & \log \det A \label{eq:obj_func} \\
    \text{s.t.} \quad     & A \text{ is positive semidefinite} \label{eq:constraint_psd} \\
                          & \forall i: \|a_{i}^{T} A\|_2 \le b_i - a_{i}^{T} c \label{eq:constraint_contained}
\end{align}

This is a maximization problem over all positive semidefinite matrices $A$ and vectors $c$. However, it is enough to maximize only over matrices and vectors of the form presented in Theorem \ref{th:alpha_beta_form}. This observation allows us to rewrite the whole problem much more concisely, in
terms of $\alpha, \beta, \gamma$ and $\delta$.

By Theorem \ref{th:eigenvals}, the objective function \eqref{eq:obj_func} may be written as $\log (\lambda_1 \cdot \lambda_{2}^{n-1} \cdot \lambda_{3}^{\binom{n}{2} - n})$, and constraint \eqref{eq:constraint_psd} may be written as $\lambda_1, \lambda_2, \lambda_3 \ge 0$. We now rewrite the set of constraints \eqref{eq:constraint_contained} via explicit computation. Indeed, there are just two types of separating hyperplanes $a_i$: Those
expressing the triangle inequalities, and the other - the upper bound of $1$ on distances. 

The inequality $x_{12} + x_{23} \ge x_{13}$ may be written as $t^T x \le 0$, where
t = $-e_{12} -e_{23} + e_{13}$. The corresponding constraint \eqref{eq:constraint_contained} is
\[\|t^T A\|_2^2 \le (0 - t^T c)^2 = \delta^2 .\]

Now, the value of $\|t^T A\|_2^2$ can be computed directly. For any $1 \le i < j \le n$ we have
\begin{equation} \label{eq:ttc}
[t^T A]_{ij} = \left\{
\begin{matrix}
    -\alpha         & ij \in \{12, 13\} \\
    \alpha - 2\beta & ij = 23 \\
    \gamma - 2\beta & i = 1, j \ge 4 \\
    -\gamma         & i \ge 2, j \ge 4
\end{matrix}
\right.    
\end{equation}

and so $\|t^T A\|_2^2 = 3\alpha^2 - 4\alpha\beta + 4(n-2)\beta^2 -4(n-3)\beta\gamma + \left[\binom{n}{2}-3\right] \gamma^2$. Combined, the constraint may be written as
\begin{equation} \label{eq:constraint_triangle}
    3\alpha^2 - 4\alpha\beta + 4(n-2)\beta^2 -4(n-3)\beta\gamma + \left[\binom{n}{2}-3\right] \gamma^2 \le \delta^2.
\end{equation}

Due to $A$'s symmetry, constraint \eqref{eq:constraint_triangle} captures all the triangle inequalities defining $\M_n$.

Similarly, the condition $x_{12} \le 1$, i.e. $e_{1,2}^T x \le 1$, translates to the constraint
\begin{equation} \label{eq:constraint_max_dist}
    \alpha^2 +2(n-2) \beta^2 + \binom{n-2}{2} \gamma^2 \le (1-\delta)^2 .
\end{equation}

This turns out to be the maximization problem presented in the lemma.
\end{proof}

We are now ready to prove Theorem \ref{th:inner_main}. The proof is mostly technical, and involves solving the optimization problem presented at lemma \ref{th:maximization_problem}.

\begin{proof}[Proof of Theorem \ref{th:inner_main} (part 1 of 2)]

We want to solve the optimization problem from lemma \ref{th:maximization_problem} using Lagrange multipliers. Let us introduce slack variables $t_1 , t_2$ for constraints \eqref{eq:final_con_1}, \eqref{eq:final_con_2} accordingly. Since $\lambda_1, \lambda_2, \lambda_3$ must be positive in an optimal solution, we may ignore conditions \eqref{eq:final_con_psd}. The problem translates to the following set of 8 Lagrange equations: 

\begin{align}
    &3\alpha^2  -4\alpha\beta  +4(n-2)\beta^2  -4(n-3)\beta\gamma  +\left[\nchoosetwo - 3\right]\gamma^2  -       \delta^2  -t_1  = 0 \\
    & \alpha^2                 +2(n-2)\beta^2                               +\binom{n-2}{2}\gamma^2  - (1 - \delta)^2  -t_2  = 0 \\
    & \frac{1}{\lambda_1} +\frac{n-1}{\lambda_2} + \frac{\binom{n}{2} - n}{\lambda_3} = \mu_1\cdot (6\alpha        -4\beta) +\mu_2 \cdot (2\alpha) \label{eq:pre_linear_1} \\
    \begin{split} \label{eq:pre_linear_2} 
        & \frac{2(n-2)}{\lambda_1}          +\frac{(n-1)(n-4)}{\lambda_2}   +\frac{-2\left[\binom{n}{2}-n\right]}{\lambda_3}  = \\
        & \quad\quad\quad\quad\quad \mu_1\cdot (-4\alpha   +8(n-2)\beta -4(n-3)\gamma)  +\mu_2 \cdot (4(n-2)\beta)
    \end{split} \\
    \begin{split} \label{eq:pre_linear_3}
        & \frac{\binom{n-2}{2}}{\lambda_1}  +\frac{-(n-1)(n-3)}{\lambda_2}  +\frac{\binom{n}{2}-n}{\lambda_3}                 = \\
        & \quad\quad\quad\quad\quad \mu_1\cdot (-4(n-3)\beta  +2\left[\binom{n}{2}-3\right]\gamma)  +\mu_2 \cdot \left(2\binom{n-2}{2}\gamma\right)
    \end{split} \\
    &\mu_1 \delta + \mu_2 \delta -\mu_2 \label{eq:t_1} = 0 \\
    &\mu_1 t_1 \label{eq:t_2} = 0 \\
    &\mu_2 t_2 \label{eq:t_3} = 0
\end{align}

Observe that there is no solution with $\mu_1 = \mu_2 = 0$. Moreover, the geometry of this problem dictates that an optimal solution has $\delta \neq 0, 1$. Combined with these facts, equations \eqref{eq:t_1}, \eqref{eq:t_2} and \eqref{eq:t_3} yield $t_1 = t_2 = 0$ and $\delta = \frac{\mu_2}{\mu_1 + \mu_2}$.

Now, it is possible to substitute equations \eqref{eq:pre_linear_1}, \eqref{eq:pre_linear_2} and \eqref{eq:pre_linear_3} with the following set of their linear combinations, and get more convenient equations:
\begin{alignat*}{3}
    2\binom{n-2}{2}&\text{eq}_\eqref{eq:pre_linear_1} & -(n-3)&\text{eq}_\eqref{eq:pre_linear_2} & +2&\text{eq}_\eqref{eq:pre_linear_3} \\
             2(n-1)&\text{eq}_\eqref{eq:pre_linear_1} & +(n-4)&\text{eq}_\eqref{eq:pre_linear_2} & -4&\text{eq}_\eqref{eq:pre_linear_3} \\
                   &\text{eq}_\eqref{eq:pre_linear_1} &      +&\text{eq}_\eqref{eq:pre_linear_2} &  +&\text{eq}_\eqref{eq:pre_linear_3}
\end{alignat*}
After this substitution, we are left with the following set of equations:

\begin{alignat}{2}
    \frac{\mu_2^2}{(\mu_1 + \mu_2)^2} &=        3\alpha^2        -4\alpha\beta   +4(n-2)\beta^2   -4(n-3)\beta\gamma  +\left[\nchoosetwo - 3\right]\gamma^2 \\
    \frac{\mu_1^2}{(\mu_1 + \mu_2)^2} &=         \alpha^2                        +2(n-2)\beta^2                                +\binom{n-2}{2}\gamma^2   \\
                        3\binom{n}{3} &= 2((3n-4) \mu_1      +(n-2)  \mu_2)      \cdot (  \alpha  -2            \beta                         +\gamma)^2 \\
                        3\binom{n}{3} &= (4(n-1)  \mu_1     +2(n-2)  \mu_2)      \cdot (  \alpha  +(n-4)        \beta                    -(n-3)\gamma)^2 \\
                               \nchoosetwo &= 2(       \mu_1           +  \mu_2)      \cdot (  \alpha  +2(n-2)       \beta           +\binom{n-2}{2}\gamma)^2
\end{alignat}

Due to geometric considerations, parameters $\lambda_1$, $\lambda_2$, $\lambda_3$ and $\delta$ are all $\Theta(1)$. Combined with Theorem \ref{th:eigenvals}, we have:
\begin{align*}
    \alpha &= \lambda_3 + o(1) = \Theta(1) \\
    \beta  &= O\left(\frac{1}{n} \right) \\
    \gamma &= O\left(\frac{1}{n^2} \right).
\end{align*}

Substituting this in our equations, we get the following asymptotic equalities:

\begin{alignat}{2}\label{eq:almost_there}
    \frac{\mu_2^2}{(\mu_1 + \mu_2)^2} &= 3\lambda_3^2 + O\left(\frac{1}{n} \right) \\
    \frac{\mu_1^2}{(\mu_1 + \mu_2)^2} &=  \lambda_3^2 + O\left(\frac{1}{n} \right) \\
                    \frac{1}{4} n^{2} &= (3\mu_1 + \mu_2) \cdot \lambda_3^2 + O(n) \\
                    \frac{1}{4} n^{2} &= (2\mu_1 + \mu_2) \cdot \lambda_2^2 + O(n) \\
                    \frac{1}{4} n^{2} &= ( \mu_1 + \mu_2) \cdot \lambda_1^2 + O(n).
\end{alignat}

The following solution then follows by simple calculation:
\begin{alignat*}{2}
    \lambda_1   &= \frac{\sqrt{3} - 1}{2} \cdot \sqrt[4]{3} + o(1)                                                            && \approx 0.4817 + o(1) \\
    \lambda_2   &= \frac{\sqrt{3} - 1}{2} \cdot \sqrt{3 - \sqrt{3}} + o(1)                                                    && \approx 0.4122 + o(1) \\
    \lambda_3   &= \frac{\sqrt{3} - 1}{2} + o(1)                                                                              && \approx 0.3660 + o(1) \\
    \alpha      &= \frac{\sqrt{3} - 1}{2} + o(1)                                                                              && \approx 0.3660 + o(1) \\
    \beta       &= \frac{\sqrt{3} - 1}{2} \cdot (\sqrt{3 - \sqrt{3}} - 1) \cdot \frac{1}{n} + o\left( \frac{1}{n}\right)      && \approx 0.0461 \cdot \frac{1}{n} + o\left( \frac{1}{n}\right) \\
    \gamma      &= \frac{\sqrt{3} - 1}{2} \cdot (4 - 2\sqrt{3 - \sqrt{3}}) \cdot \frac{1}{n^2} + o\left( \frac{1}{n^2}\right) && \approx 0.6398 \cdot \frac{1}{n^2} + o\left( \frac{1}{n^2}\right)\\
    \delta      &= \frac{\sqrt{3} - 1}{2} \cdot \sqrt{3} + o(1)                                                               && \approx 0.6340 + o(1).
\end{alignat*}
\end{proof}

We can now complete the proof of the theorem and bound the smallest 
inflation factor $r > 0$ such that $\Ellipsoid_n$ dilated by $r$ about its center 
contains all of $\M_n$. As stated in the Theorem, we show that it is between $0.99n + o(n)$ and $1.22n + o(n)$.
Note again that this bound is linear in $n$, and is better than the quadratic bound guaranteed by Theorem \ref{th:lowner_john_dilation}.

\begin{proof}[Proof of Theorem \ref{th:inner_main} (part 2 of 2)]
To obtain an upper bound, we seek a small $r>0$, such that every $x \in \M_{n}$ is also in $r \Ellipsoid_n$. 
Clearly, an even stronger condition
is the same for all $x\in {[0, 1]}^{\nchoosetwo}$, or, equivalently
the same for all $0,1$ vectors. So, let
$x \in \{0, 1\}^{\nchoosetwo}$ be a $0,1$ vector, and let
$r = {\nchoosetwo}^{\nicefrac{1}{2}} \cdot \frac{\delta_{n}}{\lambda_3} = \sqrt{3\binom{n}{2}} + o(n) \approx 1.22n + o(n)$. Then
\[{\|x-c_{n}\|}_{2}^{2} = \left(\nchoosetwo - \|x\|_{1}\right) \cdot \delta_{n}^2 + \|x\|_{1} \cdot {(1 - \delta_{n})}^2 \le \nchoosetwo \delta_{n}^{2}\]
It follows that if $n$ is large enough, then $\M_{n}\subseteq r \Ellipsoid_n$
because for every $x \in \M_{n}$ there holds
\begin{align*}
  {\|{\left(r A_{n}\right)}^{-1}(x-c_{n})\|_{2}} \le \frac{1}{r} \cdot \frac{1}{\lambda_3}\max_{x \in {\{0, 1\}}^{\nchoosetwo}}{\|x-c_{n}\|_{2}} = 1.
\end{align*}

To obtain a lower bound, consider a vertex $v$ corresponding to a cut of size
$\left\lfloor \frac{n}{2} \right\rfloor$. A direct computation shows, 
that $v$ is not in $r' \Ellipsoid_n$ for
$r' < \frac{3^{\frac{3}{4}} \sqrt{5 + \sqrt{3}}}{6}n + o(n) \approx 0.99n + o(n)$.
\end{proof}

\begin{remark}
    We do not know whether or not the inflated ellipsoid
    with this value of $r'$ contains $\M_n$, but it does contain
    the {\em Cut Polytope}.
\end{remark}

\subsection{Conclusions from Theorem \ref{th:inner_main}} \label{subsec:inner_conclusions}

The rest of this section is devoted to some conclusions from this main result.

Kozma, Meyerovitch, Peled and Samotij proved in their beautiful paper \cite{kozma2021does}
that $\M_n$ is well-approximated by the hypercube $[\frac{1}{2} - n^{-s}, 1]^{\nchoosetwo}$, 
where $s>0$ is an absolute constant (see their Theorem 1.3).
Namely, for all $n\ge 3$, a uniformly sampled metric space $d\in\M_n$, satisfies:
\[\text{with probability~}>1-O(n^{-s}), \text{all distances~} d_{ij} \text{~are~} \ge \frac{1}{2} -n^{-s}\]

This suggests the question concerning the range of the coordinates in
$\Ellipsoid_n$'s points.

\begin{corollary} \label{th:min_distance}
Let $n \geq 3$ be an integer and let $\Ellipsoid_n$ be the largest volume ellipsoid contained in $\M_n$. Then
\[
    \min_{d\in \Ellipsoid_{n}, \: ij\in \binom{[n]}{2}}{d_{ij}}
    = 2-\sqrt{3} + o(1) \approx 0.27 + o(1)
\]
\end{corollary}

\begin{proof}
Due to the $S_n$-symmetry of $\Ellipsoid_n$, there holds: 
\[\displaystyle\min_{d\in \Ellipsoid_{n}, \: ij\in \binom{[n]}{2}}{d_{ij}} = \min_{d\in \Ellipsoid_{n}}{d_{12}} = \min_{d\in \Ellipsoid_{n}}{e_{12}^{T} d}.\]
Recall that $\Ellipsoid_{n}=A_n \unitball + c_n$, whence
\begin{align*}
    \min_{d\in \Ellipsoid_n}{e_{12}^{T} d} & = \min_{\|v\|_{2} \leq 1}{e_{12}^{T}c_{n} + e_{12}^{T} A_{n}v} = \delta_{n} + \min_{\|v\|_{2} \leq 1}{\left(e_{12}^{T} A_{n}\right) v} = \\
    & = \delta_n - \|e_{12}^{T} A_{n}\|_{2} = \delta_n - \sqrt{\alpha_{n}^2 + 2(n-2)\beta_{n}^2 + \binom{n-2}{2}\gamma_{n}^2} = \\
    & = 2 - \sqrt{3} + o(1) \approx 0.27 + o(1)
\end{align*}
\end{proof}

Finally, we determine the intersection of $\Ellipsoid_n$ and $\M_n$. 
Notice that inner \lownerjohn ellipsoids have a property similar to the
one depicted in Theorem \ref{th:john_outer_contact} \cite{ball1992ellipsoids}.

\begin{theorem} \label{th:inner_contact_lemma}
For every $n \ge 3$, the intersection between $\M_n$ and its inner \lownerjohn ellipsoid 
$\Ellipsoid_n$ is comprised of the $S_n$-orbits of the following points $p, q \in \Rindices$:

\begin{equation} \label{eq:contact_triangle}
    [p]_{ij} = \left\{
    \begin{matrix}
        1 - \frac{\sqrt3}{3} + o(1)                 & \approx 0.42 + o(1) & ij \in \{12, 13\} \\
        2\left( 1 - \frac{\sqrt3}{3} \right) + o(1) & \approx 0.85 + o(1) & ij = 23 \\
        \frac{3 - \sqrt{3}}{2} + o(1)               & \approx 0.63 + o(1) & \text{otherwise}
    \end{matrix}
    \right.
\end{equation}
\begin{equation} \label{eq:tangency_diameter_bound}
    [q]_{ij} = \left\{
    \begin{matrix}
        1                                                                & ij = 12 \\
        \frac{3 - \sqrt{3}}{2} + o(1) \approx 0.63 + o(1) & \text{otherwise}
    \end{matrix} 
    \right.
\end{equation}

\end{theorem}

\begin{proof}
Recall that $\M_n$ has 2 types of facets: Those that
express the triangle inequalities, and those bounding all distances by $1$.

Let $p \in \Rindices$ be the contact point of $\Ellipsoid_n$ with
the facet corresponding to the triangle inequality $d_{23} - d_{12} - d_{23} \le 0$,
and let us denote $- e_{12} - e_{23} + e_{13}\in\Rindices$ by $t$. Since
$\Ellipsoid_{n}=A_n \unitball + c_n$, there holds:

\begin{multline*}
    p = c_n + A_n \left( \argmax_{\|v\| \le 1} {t^{T} (A_n v + c_n)} \right)
    = c_n + A_n \left(\argmax_{\|v\| \le 1} {t^{T} A_n v} \right) \\
    = c_n + \frac{1}{\|t^T A_n\|}A_n \left(A_n^T t \right)
    = c_n + \frac{1}{\|t^T A_n\|} A_n^2 t .
\end{multline*}

As shown in \eqref{eq:ttc}, for all $1 \le i, j \le n$:
\begin{equation*}
    [A t]_{ij} = \left\{
    \begin{matrix}
        \alpha          & ij \in \{12, 13\} \\
        \alpha - 2\beta & ij = 23 \\
        \gamma - 2\beta & i = 1, j \ge 4 \\
        \gamma          & i \ge 2, j \ge 4
    \end{matrix}
    \right.    
\end{equation*}
and so
\begin{equation*}
    [A \cdot At]_{ij} = \sum_{kl} A_{ij, kl}[At]_{kl} = \left(A_{ij, 12} + A_{ij, 13}\right) \alpha + A_{ij, 23} (2\beta - \alpha) + \sum_{l \ge 4} A_{ij, 1l}(2\beta - \gamma) + \sum_{\substack{k \ge 2 \\ l \ge 4}} A_{ij, kl} \gamma .
\end{equation*}

Equation \eqref{eq:contact_triangle} follows by 
direct computation (which we omit for the sake of brevity), 
combined with the bounds from
Theorem \ref{th:inner_main}, we get result. An application of $\sigma \in S_n$ to $p$
yields the intersection point of $\Ellipsoid_n$
with the hyperplane $t_{\sigma(1), \sigma(2), \sigma(3)}$.

If $q \in \Rindices$ is the point at which $\Ellipsoid_n$ intersects the facet 
$d_{12} \le 1$, then
\[
    q = c_n + A_n \left(\argmax_{\|v\| \le 1} e_{01}^T (A_n v + c_n) \right) = c_n + A_n \left(\argmax_{\|v\| \le 1} [A_n]_{12} v \right) = c_n + \frac{1}{\|[A_n ]_{12} \|_2} A_n [A_n]_{12}
\]
and equation \eqref{eq:tangency_diameter_bound} again follows
by a direct computation whose details we omit.
\end{proof}

\section{Discussion and Open Problems}

\begin{itemize}
\item 
Let us mention one not-so-obvious reason to study the approximation of convex
bodies by ellipsoids: Optimizing a linear function on an ellipsoid is a rather
trivial matter, since an ellipsoid is the affine image of a ball. It is conceivable
that this could offer a method for finding approximate solutions to hard
optimization problems. If you are unable to determine the maximum of some 
(linear or convex) function $f$ on a convex body $K$, you may first approximate $K$
by an ellipsoid $\Ellipsoid$ and optimize $f$ on it. We are presently unable
to offer any concrete examples of interest, but this leads us to our next point.
\item
Can one find the \lownerjohn ellipsoid of the Cut Polytope? How good is
the approximation that this provides to the MAXCUT problem?
Note that the symmetry arguments that we used in the study of $\M_n$'s
approximating ellipsoids apply here as well.
\item
We managed to get good upper and lower bounds on the optimal inflating factor
in Theorem \ref{th:inner_main}, but the exact number still eludes us. It remains
open to determine the exact factor.
\end{itemize}

%Bibliography

\printbibliography
\end{document}